\documentclass[12pt]{amsart}
\usepackage[dvips]{graphicx}
\usepackage{color}
\usepackage[T1]{fontenc}
\usepackage[utf8]{inputenc}
\usepackage{tikz}

\theoremstyle{definition} 
\newtheorem{definition}{Definition}[section]
\newtheorem{notation}{Notation}[section]

\theoremstyle{plain} 
\newtheorem{proposition}[definition]{Proposition}
\newtheorem{lemma}[definition]{Lemma}
 
\newtheorem{observation}[definition]{Observation}

\theoremstyle{remark} 
\newtheorem{remark}[definition]{Remark}

\newtheorem*{Reader's guide}{Reader's guide}
\thanks{The author is partially supported by ANR ``GeoDyM''}

\title[A combinatorial move]{A combinatorial move on the set of Jenkins-Strebel differentials}
\author{Corentin Boissy}
\address{Corentin Boissy\\
Aix Marseille Université, CNRS, Centrale Marseille, I2M, UMR 7373, 13453 Marseille France}
\email{corentin.boissy@univ-amu.fr}


\date{\today}

\begin{document}
\maketitle

\begin{abstract} 
We describe an elementary combinatorial move on the set of quadratic differentials with a horizontal one cylinder decomposition. Computer experiment suggests that the  corresponding equivalent classes are in one-to-one correspondence with the connected component of the strata. 
\end{abstract}

\section{Introduction}
A nonzero quadratic differential with at most simple poles on a compact Riemann surface naturally defines a flat metric with conical singularities on this surface. Geometry and dynamics on such flat surfaces, in relation to geometry and dynamics on the corresponding moduli space of Abelian differentials is a very rich topic and has been widely studied in the last 30 years. It is related to interval exchange transformations, billards in polygons, Teichmüller dynamics. 

One important case is when the quadratic differential is the square of a holomorphic one-form (Abelian differential). It was first proven by Veech \cite{Veech90} that the strata of Abelian differentials are not connected, by using a combinatorial description, usually called \emph{extended Rauzy classes}.
The connected components of the moduli space of Abelian differentials were described by Kontsevich and Zorich in \cite{KoZo}. They showed that each stratum has up to three connected component, which are described by two invariants: hyperellipticity and parity of spin structure, that arise under some conditions on the set of zeroes.
 Later, Lanneau has described the connected components of the moduli space of quadratic differentials. In this case, the nonconnected strata either contain a hyperelliptic  connected component (which is an infinite family of strata), or belong to a finite family of ``exceptionnal strata''.

The exceptionnal strata $\mathcal{Q}(-1,9), \mathcal{Q}(-1,3,6), \mathcal{Q}(-1,3,3,3)$ and $\mathcal{Q}(12)$ where initially proved non connected by the computation by Zorich of the corresponding extended Rauzy classes. More recently, Möller and Chen \cite{MC} discovered few more exceptionnal strata: $\mathcal{Q}(3,9)$, $\mathcal{Q}(6,6)$, $\mathcal{Q}(3,3,6)$ and $\mathcal{Q}(3,3,3,3)$, and gave a proof of the non-connectedness of all exceptionnal strata, by using techniques of algebraic geometry.

\section{Classes of cylindrical permutations}
\subsection{Case of Abelian differentials}
Consider a finite alphabet $\mathcal{A}$ and a vector $\zeta\in \mathbb{R}^{\mathcal{A}}$, whose entries have positive real part. Choose two total ordering of $\mathcal{A}$, \emph{i.e.} one-to-one maps $\pi_t, \pi_b:\mathcal{A}\to \{1,\dots ,d\}$. The combinatorial data $\pi=(\pi_t,\pi_b)$ will be usually written as an array with two lines of symbols in $\mathcal{A}$:
$$
\begin{pmatrix}
\pi_t^{-1}(1)&\dots &\pi_t^{-1}(d) \\
\pi_b^{-1}(1)& \dots &\pi_b^{-1}(d)
\end{pmatrix}
$$
Such $\pi$ will be refered to as a \emph{labeled permutation}. 

We construct a translation surface in the following way, that we will refer to as the \emph{cylindrical construction} (note that  \emph{is not} the Veech construction): consider the broken line $L_t$ on $\mathbb{C}=\mathbb R^2$ defined by concatenation of the vectors $\zeta_{\pi_t^{-1}(j)}$ (in this order) for $j=1,\dots,d$ with starting point at $(0,l)\in \mathbb{R}^2$. Similarly, we consider the broken line $L_b$  defined by concatenation of the vectors $\zeta_{\pi_b^{-1}(j)}$ (in this order) for $j=1,\dots,d$ with starting point at the origin. Then join the two starting points of $L_t,L_b$ a vertical segment and similarly with the two ending points.If $l$ is large enough, so that $L_t,L_b$ do not intersect, one gets a polygon whose sides come by pairs of a parallel segments of the same lenght. Gluing these segments by translations, one gets a translation surface $S(\pi,\zeta)$, with $\pi=(\pi_t,\pi_b)$ (see~Figure~\ref{fig:cylind}). 

\begin{figure}[htb]
\begin{tikzpicture}[scale=0.4]
\coordinate (b1) at (4,1);
\coordinate (b1m) at (-4,-1);
\coordinate (a) at (-1,-2);
\coordinate (am) at (1,2);
\coordinate (a0) at (1,-2);
\coordinate (am0) at (-1,2);
\coordinate (b2) at (4,-1);
\coordinate (b2m) at (-4,1);
\coordinate (b3) at (1.5,-1); 
\path (a)++(b1) coordinate (b1a);
\path (am)++(b1m) coordinate (b1am);
\path (a)++(b2) coordinate (b2a);
\path (am)++(b2m) coordinate (b2am);

\draw [fill=gray!20] (0,0)--++(0,6) --++(4,-1) node[midway, above] {$\zeta_a$}--++(b1)  node[midway, above] {$\zeta_b$}--++(a0)  node[midway, above right] {$\zeta_c$}--++(10,2)  node[midway, above] {$\zeta_d$}--++(b2)  node[midway, above] {$\zeta_e$}--++(0,-6)--++(-4,1)  node[midway, below] {$\zeta_a$}--++(b1m)  node[midway, below] {$\zeta_b$}--++(am0) node[midway, below left] {$\zeta_c$}--++(-10,-2) node[midway, below] {$\zeta_d$}--++(b2m) node[midway, below] {$\zeta_e$}--cycle;
\end{tikzpicture}
\caption{Example of cylindrical construction} 
\label{fig:cylind}
\end{figure}
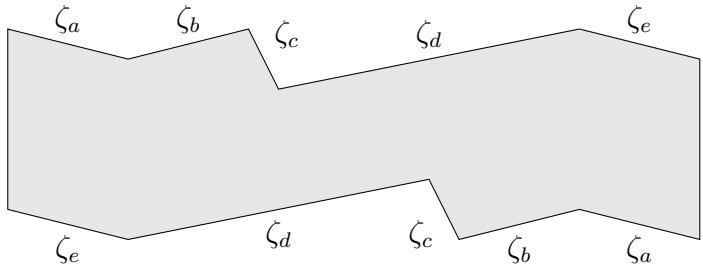

If $\zeta$ has zero imaginary part, $S(\pi,\zeta)$ is a translation surface with a one cylinder decomposition.

Now we choose any continuous data $\zeta\in \mathbb{C}^\mathcal{A}$ such that $Im(\zeta_\alpha)<0$ we deform the $S$ by changing the real part of $\zeta_\alpha$  until we have $Re(\zeta_\alpha)<0$ (see Figure~\ref{fig:ab}). Keeping the other parameter unchanged and preserving the identifications, we still have a translation surface, that belongs to the same connected component as $S$ (if $|Re(\zeta_\alpha)|$ is small enough). This surface is not any more presented as in the previous construction. However, after some suitable cutting and pasting, it is the case, with another combinatorial datum.

We assume that $\alpha$ is not the leftmost element of the top line component, nor the rightmost element of the  bottom line. Let $\beta_1\in \mathcal{A}$ be the index that correspond to the segment immediately on the left of $\alpha$ in $L_t$ and $\beta_2\in \mathcal{A}$ by the one immediately on the right of $\alpha$ in $L_b$. 
We now cut the triangle that contain the segments corresponding to $\alpha$ and ${\beta_1}$ on the top, and glue it along the segment corresponding to $\beta_1$ on the bottom. Similarly, we cut  the triangle in the bottom  correponding to $\alpha,\beta_2$ and glue it to the top along the segment corresponding to $\beta_2$. 
So the flat surface is obtained by the same  construction as before, where the combinarorial data
$$
\left(\begin{array}{l}
  \dots \beta_1 \alpha \dots \dots \beta_2\dots  \\
  \dots  \alpha \beta_2 \dots \dots  \beta_1\dots 
\end{array}\right)$$
as been replaced by 
$$
\left(\begin{array}{l}  \dots \beta_1 \dots \dots \alpha \beta_2\dots  \\
  \dots   \beta_2 \dots \dots  \beta_1 \alpha\dots 
\end{array}
\right)$$
 and where $\zeta$ as been replaced by $\zeta'$ such that 
 \begin{eqnarray*}
 \zeta'_\alpha&=&-\zeta_\alpha \\
 \zeta'_{\beta_1}&=&\zeta_{\beta_1}+\zeta_\alpha \\
 \zeta'_{\beta_2}&=&\zeta_{\beta_2}+\zeta_\alpha 
 \end{eqnarray*}


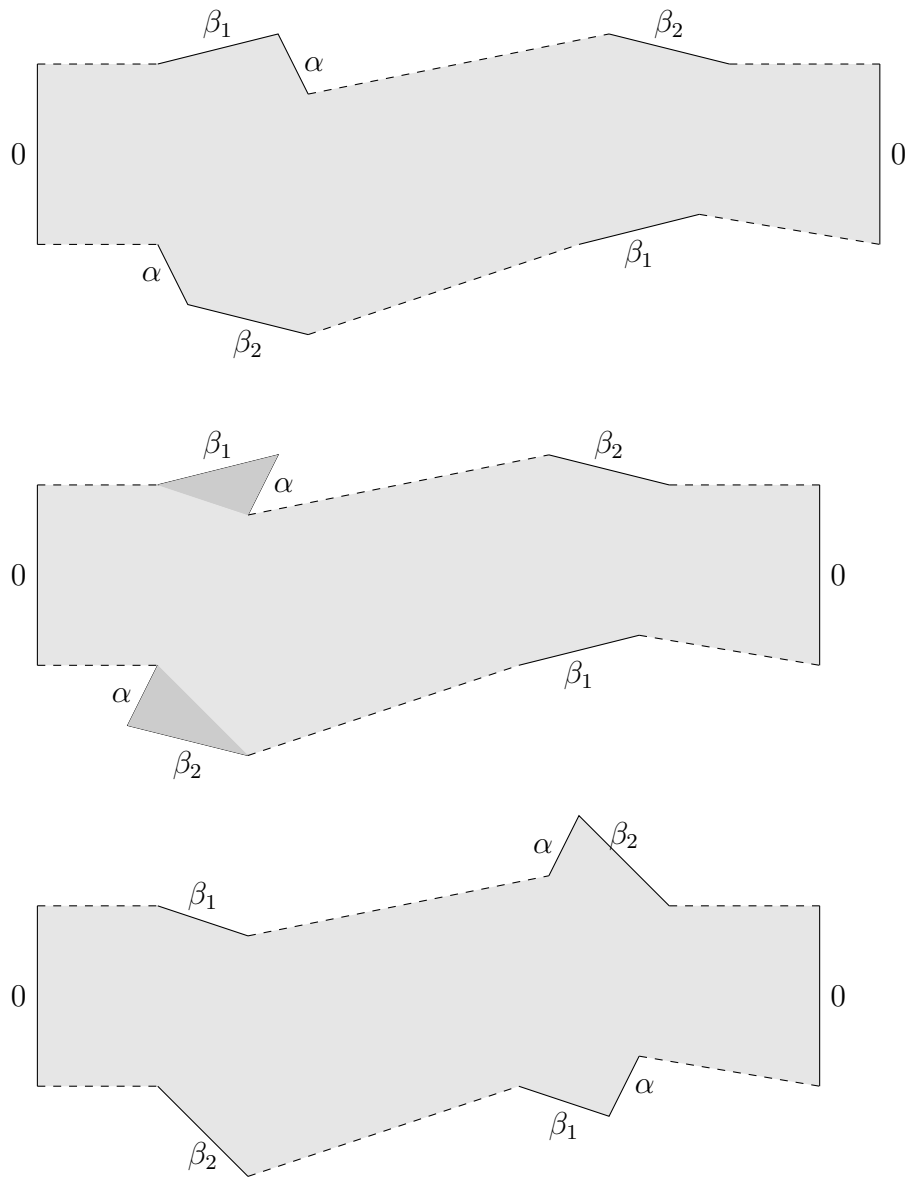
\begin{figure}[htb]
\label{move}
\begin{tikzpicture}[scale=0.4]
\coordinate (b1) at (4,1);
\coordinate (b1m) at (-4,-1);
\coordinate (a) at (-1,-2);
\coordinate (am) at (1,2);
\coordinate (a0) at (1,-2);
\coordinate (am0) at (-1,2);
\coordinate (b2) at (4,-1);
\coordinate (b2m) at (-4,1);
\coordinate (b3) at (1.5,-1); 
\path (a)++(b1) coordinate (b1a);
\path (am)++(b1m) coordinate (b1am);
\path (a)++(b2) coordinate (b2a);
\path (am)++(b2m) coordinate (b2am);

\fill [fill=gray!20] (0,14)--++(0,6)--++(4,0)--++(b1)--++(a0)--++(10,2)--++(b2)--++(5,0)--++(0,-6)--++(-6,1)--++(b1m)--++(-9,-3)--++(b2m)--++(am0)--cycle;
\draw (0,14)--++(0,6) node[midway, left] {$0$};
\draw [dashed] (0,20)--++ (4,0);
\draw (4,20)--++(b1) node[midway, above] {$\beta_1$}--++(a0) node[midway, right] {$\alpha$};
\draw [dashed] (7+2,5+14)--(17+2,7+14);
\draw (17+2,7+14)--++(b2) node[midway, above] {$\beta_2$};
\draw [dashed] (21+2,6+14)--(26+2,6+14);
\draw (26+2,6+14)--(26+2,14) node[midway, right] {$0$};
\draw [dashed] (0,14)--(4,14);
\draw (4,14)--++(a0) node[midway, left] {$\alpha$}--++(b2) node[midway, below] {$\beta_2$};
\draw [dashed] (7+2,-3+14)--(16+2,14);
\draw (16+2,14)--++ (b1) node[midway, below] {$\beta_1$};
\draw [dashed] (20+2,15)--(26+2,14);

\fill [fill=gray!20] (0,0)--++(0,6)--++(4,0)--++(b1)--++(a)--++(10,2)--++(b2)--++(5,0)--++(0,-6)--++(-6,1)--++(b1m)--++(-9,-3)--++(b2m)--++(am)--cycle;
\draw (0,0)--++(0,6) node[midway, left] {$0$};
\draw [dashed] (0,6)--++ (4,0);
\draw (4,6)--++(b1) node[midway, above] {$\beta_1$}--++(a) node[midway, right] {$\alpha$};
\draw [dashed] (7,5)--(17,7);
\draw (17,7)--++(b2) node[midway, above] {$\beta_2$};
\draw [dashed] (21,6)--(26,6);
\draw (26,6)--(26,0) node[midway, right] {$0$};
\draw [dashed] (0,0)--(4,0);
\draw (4,0)--++(a) node[midway, left] {$\alpha$}--++(b2) node[midway, below] {$\beta_2$};
\draw [dashed] (7,-3)--(16,0);
\draw (16,0)--++ (b1) node[midway, below] {$\beta_1$};
\draw [dashed] (20,1)--(26,0);

\fill [fill=gray!40] (4,6)--++(b1) --++ (a)--cycle;
\fill [fill=gray!40] (4,0)--++(a) --++ (b2)--cycle;

\fill [fill=gray!20] (0,-14)--++(0,6)--++(4,0)--++(b1a)--++(10,2)--++(am)--++(b2a)--++(5,0)--++(0,-6)--++(-6,1)--++(a)--++(b1am)--++(-9,-3)--++(b2am)--cycle;
\draw (0,0-14)--++(0,6) node[midway, left] {$0$};
\draw [dashed] (0,6-14)--++ (4,0);
\draw (4,6-14)--++(b1a) node[midway,above] {$\beta_1$};
\draw [dashed] (7,5-14)--(17,7-14);
\draw (17,7-14)--++(am) node[midway, left] {$\alpha$}--++(b2a) node[midway, above] {$\beta_2$};
\draw [dashed] (21,6-14)--(26,6-14);
\draw (26,6-14)--(26,0-14) node[midway, right] {$0$};
\draw [dashed] (0,0-14)--(4,0-14);
\draw (4,0-14)--++ (b2a) node[midway, below] {$\beta_2$};
\draw [dashed] (7,-3-14)--(16,0-14);
\draw (16,0-14)--++ (b1a) node[midway, below] {$\beta_1$}--++(am) node[midway, right] {$\alpha$};
\draw [dashed] (20,1-14)--(26,0-14);

\fill [fill=gray!40] (4,6)--++(b1) --++ (a)--cycle;
\fill [fill=gray!40] (4,0)--++(a) --++ (b2)--cycle;

\end{tikzpicture}
\caption{Changing the combinatorial datum for the cylindrical construction}
\label{fig:ab}
\end{figure}

\begin{notation}
If $\pi$ is the initial combinatorial datum we will denote the new combinatorial datum by $T_\alpha(\pi)$.
\end{notation}
One can also do an analogous procedure by starting from $\zeta_\alpha$ with positive imaginary part. One gets the transformation
$$
\left(\begin{array}{l}  \dots \beta_1 \dots \dots \alpha \beta_2\dots  \\
  \dots   \beta_2 \dots \dots  \beta_1 \alpha\dots 
\end{array} \right)
\to
\left(\begin{array}{l}
  \dots \beta_1 \alpha \dots \dots \beta_2\dots  \\
  \dots  \alpha \beta_2 \dots \dots  \beta_1\dots 
\end{array}\right)  
$$
\emph{i.e.}, one gets $T_\alpha^{-1}(\pi)$

The following proposition is clear from the definition of $T_\alpha$.
\begin{proposition}
$\pi$ and $T_\alpha(\pi)$ corresponds to translation surfaces in the in the same connected component of the same stratum, when $T_\alpha(\pi)$ is well defined.
\end{proposition}

There are some other simple transformations that preserve the connected component.
$$
U_t:\left(\begin{array}{l}   \alpha_1 \alpha_2 \dots \alpha_n \\
 \beta_1 \beta_2  \dots  \beta_n  
\end{array}\right)
\mapsto
\left(\begin{array}{l}   \alpha_2  \dots \alpha_n \alpha_1 \\
 \beta_1 \beta_2  \dots  \beta_n  
\end{array}\right)
$$

$$
U_b:\left(\begin{array}{l}   \alpha_1 \alpha_2 \dots \alpha_n \\
 \beta_1 \beta_2  \dots  \beta_n  
\end{array}\right)
\mapsto
\left(\begin{array}{l}    \alpha_1 \alpha_2 \dots \alpha_n \\
 \beta_2  \dots   \beta_n  \beta_1
\end{array}\right)
$$
and 
$$S:
\left(\begin{array}{l}   \alpha_1 \alpha_2 \dots \alpha_n \\
 \beta_1 \beta_2  \dots  \beta_n  
\end{array}\right)\mapsto
\left(\begin{array}{l}   \beta_n \dots\beta_2 \beta_1 \\
\alpha_n \dots \alpha_2 \alpha_1  
\end{array}\right)
$$

A renumbering of a labeled permutation is obtained by precomposing the maps $\pi_t,\pi_b$ with a permutation of the alphabet $\mathcal{A}$. For instance, the labeled permutation $\left(\begin{smallmatrix} b &c &d&a \\ a&d&c&b \end{smallmatrix}\right)$ is a renumbering of the labeled permutation $\left(\begin{smallmatrix} a&b &c &d \\ d&c&b&a \end{smallmatrix}\right)$. Of course, the set of labeled permutation on a alphabet $\mathcal{A}$ up to renumbering is naturally identified with the set of permutations of $\mathcal{A}$.

Computing equivalent classes for these action, one gets the following experimental result,  obtained by computer experiment\footnote{The SAGE programs can be found on the webpage of the author.}.
\begin{observation}
Two labeled permutations with up to 12 letters representing the same connected component for the cylindrical  construction in the moduli space of Abelian differentials, can be joined by a succession of moves $T_\alpha,T_\alpha^{-1}, U_t,U_b$, up to renumbering.
\end{observation}


\subsection{Case of quadratic differentials}
One can define quadratic differentials in a similar way, by using \emph{generalized permutation}. We will define a generalized permutation of type $(l,m)$ as an array of two lines of symbols in a finite alphabet $\mathcal{A}$:
$$\pi=
\left(
\begin{array}{lll}
\alpha_1&\dots& \alpha_l \\
\beta_1&\dots& \beta_m
\end{array}
\right)$$
with the condition that each letter  appears exactly two times in the list $(\alpha_1,\dots ,\alpha_l,\beta_1,\dots \beta_m)$. 
We also assume that there exists $1\leq i,j\leq l$ such that $\alpha_i=\alpha_j$ and $1\leq i',j'\leq m$ such that $\beta_{i'}=\beta_{j'}$, \emph{i.e.} there is at least a symbol that appear twice on the top line (and therefore, not on the bottom line) and there is at least a symbol that appear twice on the bottom line (and not on the top line).

Then, choosing a vector $\zeta\in \mathbb{C}^{\mathcal{A}}$ whose entries have positive real part, and with the extra condition that
$\sum_{i=1}^l \zeta_{\alpha_i}=\sum_{j=1}^m \zeta_{\beta_j}$
one gets a flat surface be a similar construction as before. The resulting surface is a half-translation surface, \emph{i.e.} a Riemann surface endowed with a quadratic differential which is not the square of an Abelian differential. As before, if all the $\zeta_i$ have zero imaginary part, the corresponding flat surface has a one cylinder decomposition.

\begin{definition}
Let $\pi$ be a generalized permutation. A letter $\alpha\in \mathcal{A}$ is said to be of type \emph{top-bottom} if it appears both on the top and the bottom line. Similarly, it is of type \emph{top-top} (resp. \emph{bottom-bottom}) if it appears only on the top (resp. bottom).
\end{definition}

We will define an analogous move as for the case of Abelian differential. However, the corresponding combinatorial operation depends on the type of the involved letters $\alpha,\beta_1,\beta_2$ (and the definition of $\beta_1,\beta_2$ depends on the type of $\alpha$\dots ). However, by writing $\pi$ in a slightly different way, one gets a unified combinatorial operation.

\begin{definition}
Let $\pi=\left(\begin{smallmatrix} \alpha_1,\dots ,\alpha_l \\ \beta_1\dots ,\beta_m\end{smallmatrix}\right)$. We define the \emph{one line represention} of $\pi$ to be
$$\alpha_1\alpha_2\dots \alpha_l | \beta_m \beta_{m-1}\dots \beta_1 $$
\end{definition}

\begin{definition}\label{def:transfo}
Let $\pi$ be a generalized permutation, and $\alpha\in \mathcal{A}$ with the following condition:
\begin{itemize}
\item $\alpha$ is either of type top-bottom, or is not the only top-top (resp. bottom-bottom) letter  (such $\alpha$ is called \emph{regular}).
\item Each occurence of $\alpha$ admits a letter on its left in the one line representation of $\pi$ (\emph{i.e.} no occurence of $\alpha$ not the leftmost letter, or follows the symbol ``$|$'').
\end{itemize}
Let $\pi$ be written in a one line representation as (up to changing the order of the four blocks $\{\beta_1\alpha, \beta_2\alpha, \beta_1, \beta_2 \}$):
$$\dots \beta_1\alpha \dots \beta_2\alpha\dots \beta_1 \dots \beta_2 \dots $$
Then we define $T_\alpha(\pi)$ to be the generalized permutation:
$$\dots  \beta_1\dots \beta_2 \dots  \alpha\beta_1  \dots \alpha \beta_2 \dots $$
\end{definition}

One has the following proposition:
\begin{proposition}
Let $\alpha$ that satisfies the condition of Definition~\ref{def:transfo}. Then, $\pi$ and $T_\alpha(\pi)$ corresponds to the same connected component of the moduli space of quadratic differentials.
\end{proposition}

\begin{proof}
The proof is made by doing similar surgeries as in the Abelian case. 
We assume that $\alpha$ is  of type top-bottom,  its top occurence is not the leftmost element of the (top) line, and its bottom occurence is not the rightmost of the (bottom) line. Let $\beta_1$ be the letter immediatelly on the left of $\alpha$ on the top line, and $\beta_2$ be the letter immediatelly on the right of $\alpha$ on the bottom line.
 Choose a continuous parameter $\zeta$ such that $Im(\zeta_\alpha)<0$. And change   continuously the real part of $\zeta_\alpha$ so that it becomes negative. Then, as before, we cut the triangle of sides corresponding to $\beta_1\alpha$, and glue it by along the side corresponding to the other occurence of $\beta_1$. If $\beta_1$ is on the bottom, the gluing is done by a translation, otherwise it is done by an half turn. We do a similar cutting and pasting for the triangle whose sides correspond to $\alpha\beta_2$.
 
 If  $\beta_1,\beta_2$ are both of type top-bottom, the combinatorial datum is changed in exactly the same way as before. Assume for instance that $\beta_1$ is top-top, and $\beta_2$ is top-bottom. Then, it is easy to see that $\pi$ is changed in the following way:
 
 $$
\pi=\left(\begin{array}{l}
  \dots \beta_1 \alpha \dots \beta_1 \dots \beta_2\dots  \\
  \dots  \alpha \beta_2 \dots \dots  \dots \dots \dots 
\end{array}\right) \to
\pi'=\left(\begin{array}{l}  \dots \beta_1 \dots \alpha\beta_1 \dots \alpha \beta_2\dots  \\
  \dots   \beta_2 \dots \dots  \beta_1\dots \dots \dots 
\end{array}
\right)$$
The condition on $\alpha$ and the generalized permutation $\pi'$,  when written in a one line representation, corresponds precisely to Definition~\ref{def:transfo} and $T_\alpha(\pi)$.

When $\alpha$ is top-top (resp. bottom-bottom), we can define analogous transformation.  We want to have $\zeta_\alpha$ with negative real part and all other $\zeta_i$ with positive real part. Recall that there is a linear condition $(1)$ on $\zeta$ that is needed in the definition of the surface. This condition forbid $\alpha$ to be the only top-top letter.
\end{proof}

The maps $U_t, U_b$ and $S$ are defined analogously as in the case of Abelian differentials.

\begin{definition}
A generalized permutation is \emph{mixed} if it contains top-bottom, top-top and bottom-bottom letter.
\end{definition}

\begin{lemma}
Any connected component contain a one cylinder surface defined by a mixed generalized permutation.
\end{lemma}

\begin{proof}
We start from a flat surface $S$ in a connected component $\mathcal{C}$ with a one cylinder decomposition (such surface always exists, by \cite{La:cc}, Theorem~3.6), and assume that it is not mixed. Each singularity appear on one of the horizontal side, but not both. Consider one horizontal side. If there are more than two singularities,  and more than one saddle connection, then there exists a saddle connection joining a singularity $P_1$, of order $k_1$ to a singularity $P_2$ of order $k_2$, and we can shrink it to obtain a singularity of order $k_1+k_2$ on a surface $S'$. We obtain a connected component with less singularities. Assume there exists a cylinder as in the lemma, then we can continuously change $S'$, until it is that cylinder. Breaking up continously the singularity of order $k_1+k_2$ along this path (see \cite{KoZo} or \cite{La:cc} for the precise definition of breaking up a singularity), we get a continuous path in $\mathcal{C}$, whose endpoint if the required cylinder.

The former procedure cannot be realized if each sides contains only one singularity or is composed of only one saddle connection (\emph{i.e.} contains exactly two simple poles). So we only need to check the result for such surfaces.

We use Rauzy induction (see \cite{BL09}, for the definition of the Rauzy moves $\mathcal{R}_t, \mathcal{R}_b$). Assume for instance that one side is a pair of poles. The generalized permutation for the cylinder decomposition is of the form
$$\begin{pmatrix}
*\ *\ * \\
a\ a
\end{pmatrix}$$
The corresponding generalized permutation for the Veech construction is 
$$\begin{pmatrix}
0\ *\ *\ *\\
a\ a\ 0
\end{pmatrix}$$
If ``$***$'' is not of the form $a_1\dots a_n a_1\dots a_n$, then using the Rauzy move $\mathcal{R}_b$, one gets
$$\begin{pmatrix}
0\dots b\dots b\dots c\dots c\\
aa0
\end{pmatrix}$$
so that there is an element $``d''$ between the two ``c'' whose other occurence is not between the two ``c''.  Then, by applying suitably the Rauzy induction, one gets

\begin{eqnarray*}
 & &\begin{pmatrix}
0\dots b\dots b\dots c\dots d \dots c\\
aa0
\end{pmatrix} \stackrel{\mathcal{R}_t}{\to} 
\begin{pmatrix}
0\dots b\dots b\dots 0c\dots d\dots  c\\
aa
\end{pmatrix} \\
&&\stackrel{\mathcal{R}_b^k}{\to} 
\begin{pmatrix}
0\dots b\dots b\dots \\
c\dots d\dots c0aa
\end{pmatrix} \stackrel{\mathcal{R}_t^2}{\to} 
\begin{pmatrix}
0\dots b\dots b\dots\\
c\dots d\dots c 0
\end{pmatrix}
\end{eqnarray*}

Then, the element $``d''$, is top-bottom. If the top line was initially $a_1\dots a_na_1\dots a_n$, then the corresponding connected component was a hyperelliptic one. In such connected component, one can find mixed generalized permutations (for the cylindrical construction) of the kind
$$\begin{pmatrix}
a_1\dots a_n \alpha b_1\dots b_m \alpha \\
\beta b_m\dots b_1 \beta a_n\dots a_1
\end{pmatrix}$$

The other cases are left to the reader.
\end{proof}

We now can state the key result, which is proved by computer experiment\footnote{The SAGE programs can be found on the webpage of the author.}.
\begin{observation}
Two mixed generalized permutations of at most 10 letter reprensenting the same connected component can be joined by a succession of moves $T_\alpha,T_\alpha^{-1}, U_t,U_b$ and $S$, up to renumbering.
\end{observation}

\begin{remark}
In particular, we ``see'' all the connected components of the exceptionnal strata  $\mathcal{Q}(-1,9), \mathcal{Q}(-1,3,6)$, $\mathcal{Q}(-1,3,3,3)$, $\mathcal{Q}(12)$, $\mathcal{Q}(3,9)$, $\mathcal{Q}(3,3,6)$, $\mathcal{Q}(6,6)$ and $\mathcal{Q}(3,3,3,3)$.
\end{remark}

\begin{remark}
In fact, the operation $S$ is not necessary, except the for strata $\mathcal{Q}(-1,5)$, $\mathcal{Q}(-1,1,4)$ and $\mathcal{Q}(-1,9)$, for which avoiding $S$ leads to one extra class.
\end{remark}

\nocite*
\bibliographystyle{plain}
\bibliography{biblio}

\end{document}